\DocumentMetadata{uncompress}
\documentclass[11pt]{amsart}

\usepackage{amsfonts,graphics,amsmath,amsthm,amsfonts,amscd,amssymb,amsmath,latexsym,multicol,mathrsfs}
\usepackage{epsfig,url}
\usepackage{flafter}
\usepackage{fancyhdr}
\usepackage{hyperref}
\usepackage[dvipsnames]{xcolor}
\usepackage{hyperref}
\hypersetup{
	colorlinks=true,
	linkcolor=black,
	citecolor=ForestGreen,
	filecolor=cyan,
	urlcolor=teal
}

\usepackage{graphicx}
\usepackage{xcolor}
\usepackage{float}
\usepackage{bm}
\usepackage{enumitem}
\usepackage{multirow}
\usepackage{scalerel,stackengine}
\usepackage{stmaryrd}
\usepackage{datetime}
\usepackage[title]{appendix}
\usepackage{url}
\usepackage{tabulary}
\usepackage{booktabs}
\usepackage{tikz}
\usetikzlibrary{positioning}
\usepackage{verbatim}

\usepackage[utf8]{inputenc}
\usepackage[english]{babel}
\usepackage{csquotes}
\usepackage{comment}
\usepackage[
backend=biber,
style=alphabetic,
sorting=anyt,
doi=false,isbn=false,url=false,
maxalphanames=99,maxbibnames=99,backref=true,
]{biblatex}
\DefineBibliographyStrings{english}{
  backrefpage  = {\hspace{-0.35em}},
  backrefpages = {\hspace{-0.35em}},
}

\DeclareFieldFormat{bracketswithperiod}{\mkbibbrackets{#1}}

\renewbibmacro*{pageref}{%
  \iflistundef{pageref}
    {}
    {\printtext[bracketswithperiod]{%
       \ifnumgreater{\value{pageref}}{1}
         {\bibstring{backrefpages}\ppspace}
         {\bibstring{backrefpage}\ppspace}%
       \printlist[pageref][-\value{listtotal}]{pageref}}%
     }}

 \theoremstyle{plain}
 \newtheorem{theorem}[subsection]{Theorem}

 \newtheorem{proposition}[subsection]{Proposition}

 \theoremstyle{definition}

 \theoremstyle{remark}

 \theoremstyle{plain} 
\newcommand{\thistheoremname}{}
\newtheorem{genericthm}[subsection]{\thistheoremname}

  \newtheorem*{genericthm*}{\thistheoremname}
\newenvironment{namedthm*}[1]
  {\renewcommand{\thistheoremname}{#1}%
   \begin{genericthm*}}
  {\end{genericthm*}}
 
\newcommand{\spec}{\operatorname{Spec}}

\newcommand{\hilb}{\operatorname{Hilb}}

\newcommand{\supp}{\operatorname{Supp}}
\newcommand{\galois}{\operatorname{Gal}}
\newcommand{\sing}{\operatorname{sing}}

\newcommand{\sdeg}{\operatorname{sdeg}}
\newcommand{\ssdeg}{\operatorname{ssdeg}}

\newcommand{\Mod}[1]{\ (\mathrm{mod}\ #1)}

\newcommand{\nor}{\operatorname{nor}}

\newcommand{\bdiv}{\operatorname{b}}

\newcommand{\mni}{\medskip\noindent}

\newcommand{\mbb}{\mathbb}
\newcommand{\QQ}{\mbb{Q}}
\newcommand{\NN}{\mbb{N}}
\newcommand{\ZZ}{\mbb{Z}}

\newcommand{\RR}{\mbb{R}}
\newcommand{\AAA}{\mbb{A}}
\newcommand{\PP}{\mbb{P}}

\newcommand{\mc}{\mathcal}

\newcommand{\mcO}{\mc{O}}

\newcommand{\mf}{\mathfrak}
\newcommand{\mfm}{\mathfrak{m}}
\newcommand{\mb}{\mathbf}

\newcommand{\wht}{\widehat}

\newcommand{\wh}{\widehat}

\newcommand{\ol}{\overline}

\usepackage{graphicx}
\usepackage{xcolor}
\usepackage{float}
\usepackage{amssymb}
\usepackage{amsmath}
\usepackage{mathrsfs}
\usepackage{bm}
\usepackage[all,cmtip]{xy}
\usepackage{enumitem}
\usepackage[utf8]{inputenc}
\usepackage[english]{babel}
\usepackage{multirow}
\usepackage{mathtools}
\usepackage{scalerel,stackengine}
\usepackage{stmaryrd}
\usepackage{datetime}
\usepackage[title]{appendix}
\usepackage{url}
\usepackage{tabulary}
\usepackage{booktabs}
\usepackage{tikz}
\usetikzlibrary{positioning}
\usepackage{verbatim}
\usepackage{braket}

\DeclareFieldFormat{postnote}{#1}
\DeclareFieldFormat{multipostnote}{#1}

\makeatletter
\newcommand*{\da@rightarrow}{\mathchar"0\hexnumber@\symAMSa 4B }
\newcommand*{\da@leftarrow}{\mathchar"0\hexnumber@\symAMSa 4C }
\newcommand*{\xdashrightarrow}[2][]{%
  \mathrel{%
    \mathpalette{\da@xarrow{#1}{#2}{}\da@rightarrow{\,}{}}{}%
  }%
}
\newcommand{\xdashleftarrow}[2][]{%
  \mathrel{%
    \mathpalette{\da@xarrow{#1}{#2}\da@leftarrow{}{}{\,}}{}%
  }%
}
\newcommand*{\da@xarrow}[7]{%
  \sbox0{$\ifx#7\scriptstyle\scriptscriptstyle\else\scriptstyle\fi#5#1#6\m@th$}%
  \sbox2{$\ifx#7\scriptstyle\scriptscriptstyle\else\scriptstyle\fi#5#2#6\m@th$}%
  \sbox4{$#7\dabar@\m@th$}%
  \dimen@=\wd0 %
  \ifdim\wd2 >\dimen@
    \dimen@=\wd2 %
  \fi
  \count@=2 %
  \def\da@bars{\dabar@\dabar@}%
  \@whiledim\count@\wd4<\dimen@\do{%
    \advance\count@\@ne
    \expandafter\def\expandafter\da@bars\expandafter{%
      \da@bars
      \dabar@ 
    }%
  }%
  \mathrel{#3}%
  \mathrel{%
    \mathop{\da@bars}\limits
    \ifx\\#1\\%
    \else
      _{\copy0}%
    \fi
    \ifx\\#2\\%
    \else
      ^{\copy2}%
    \fi
  }%
  \mathrel{#4}%
}
\makeatother

\title{\large S\MakeLowercase{tein degree on non}-F\MakeLowercase{ano type fibrations}}
\thanks{2020 MSC:
13B40, 
14B05, 
14D06, 
14J20, 
14J27, 
14J45. 
}
\author{\large C\MakeLowercase{aucher} B\MakeLowercase{irkar} \MakeLowercase{and} \large S\MakeLowercase{antai} Q\MakeLowercase{u}}
\date{\today}

\begin{document}

\begin{abstract}
    We construct examples showing that Stein degree of vertical divisors 
    on non-Fano type log Calabi-Yau fibrations is unbounded.
\end{abstract}

\maketitle

\tableofcontents

\addtocontents{toc}{\protect\setcounter{tocdepth}{1}}


\section{Introduction}

We work over an algebraically closed field $\mbb K$ of characteristic zero unless stated otherwise.
The goal of this paper is to 
establish the unboundedness of Stein degree for vertical divisors on
non-Fano type log Calabi-Yau fibrations; see Theorem~\ref{exa-cor-d-n-thm-in-intro}.


\mni
{\textbf{\sffamily{Stein degree on log Calabi-Yau fibrations.}}}
Stein degree measures the number of connected components of general fibres of a projective morphism.
Specifically, let $S\to Z$ be a projective morphism between varieties, and
let $S\to V\to Z$ be the Stein factorisation.  
The \emph{Stein degree} of $S$ over $Z$ is defined to be 
\[ \sdeg (S/Z) \coloneqq \deg (V/Z). \]
If $S\to Z$ is not surjective, this degree is 0 by convention; see \cite[\S 3]{B-moduli}.
In this case,
the \emph{strong Stein degree} of $S$ over $Z$ is defined to be 
\[ \ssdeg (S/Z) \coloneqq \sdeg(S/C), \]
where $C$ is the image of $S$ in $Z$.  

A \emph{log Calabi-Yau fibration} $(X, B)\to Z$ consists of a log canonical (lc) pair $(X, B)$ and 
a contraction $X\to Z$ of varieties such that $K_X + B \sim_{\RR} 0/Z$; see \cite[\S 1]{B-moduli}.
Boundedness of Stein degree of lc centres on log Calabi-Yau fibrations is a
crucial ingredient for the existence of moduli space of 
semi-log canonical (slc) stable minimal models \cite{B-moduli}.

\begin{theorem}[\protect{\cite[Theorem~1.15]{B-moduli}}]\label{B-moduli-bnd-S-degree}
    Let $d\in \NN$.  Let $(X, B)\to Z$ be a log Calabi-Yau fibration of dimension $d=\dim X$.
    Then $\sdeg (I/Z)$ is bounded from above depending only on $d$ for every horizontal$/Z$ lc centre 
    $I$ of $(X, B)$.
\end{theorem}

In light of the importance of Stein degree, the first author formulates
several questions on boundedness of Stein degree of divisors on log Calabi-Yau fibrations;
see \cite[\S 1]{B-moduli}.  We have confirmed \cite[Conjecture~1.16]{B-moduli}
in the setting of generalised pairs \cite{birkar-Q-25}.
A \emph{generalised log Calabi-Yau fibration} $(X, B + \mb{M}) \to Z$ consists of a generalised log canonical 
generalised pair $(X, B + \mb{M})$ and a contraction of varieties $X\to Z$ 
such that $K_X + B + \mb{M}_X \sim_{\RR}0/Z$.
For conventions of $\bdiv$-divisors and singularities of generalised pairs, see
\cite{Corti-3-folds, B-Zhang, B-Fano} (see also \cite[\S 2]{birkar-Q-25}).

\begin{theorem}[\protect{\cite[Theorem~1.3]{birkar-Q-25}}]\label{dim-d-t}
    Let $d\in \NN$ and let $t\in (0,1]$.
	Let $(X, B + \mb{M})\to Z$ be a generalised log Calabi-Yau fibration 
	of relative dimension $d$, that is, a general fibre of $X\to Z$ has dimension $d$.
    Let $S$ be a horizontal$/Z$ irreducible component of $B$ 
	whose coefficient in $B$ is $\ge t$.
	Then $\sdeg (S^{\nor}/Z)$ is bounded from above depending only on $d,t$, 
    where $S^{\nor}$ is the normalisation of $S$.
\end{theorem}

Recall that a contraction of varieties $X\to Z$ is called \emph{of Fano type} 
if there exists a big$/Z$ $\RR$-divisor $\Gamma$ such that $(X, \Gamma)$ is Kawamata log terminal (klt)
and $K_X + \Gamma \sim_{\RR}0/Z$.
A \emph{Fano type generalised log Calabi-Yau fibration} (or a \emph{Fano type fibration} for short) 
is a generalised log Calabi-Yau fibration $(X, B + \mb{M})\to Z$, where $X$ is of Fano type over $Z$.
A key ingredient in the proof of Theorem~\ref{dim-d-t} is the following boundedness
result on strong Stein degree of vertical divisors on Fano type fibrations.

\begin{theorem}[\protect{\cite[Theorem~1.4]{birkar-Q-25}}]\label{vertical-bnd-intro}
    Let $d\in \NN$ and let $t \in (0,1]$.  Let 
    $(X, B + \mb{M})\to Z$ 
    be a Fano type generalised log Calabi-Yau fibration
    of relative dimension $d$.  
    Let $S$ be a vertical$/Z$ irreducible component of $B$ such that
    \begin{itemize} 
        \item [\emph{(1)}] the coefficient of $S$ in $B$ is $\ge t$, and
        \item [\emph{(2)}] the image of $S$ in $Z$ has codimension one in $Z$.
    \end{itemize}
    Then $\ssdeg (S^{\nor}/Z)$ is bounded from above depending only on $d, t$,
    where $S^{\nor}$ is the normalisation of $S$.
\end{theorem}

Note that Theorem~\ref{vertical-bnd-intro} also answers a question proposed on \cite[page 8]{B-moduli},
which asks the boundedness of strong Stein degree of vertical divisors on log Calabi-Yau fibrations.


\medskip

\mni
{\textbf{\sffamily{Main result.}}} 
In this paper, we give examples showing that 
the strong Stein degree of vertical divisors is unbounded if $X\to Z$ is not of Fano type in
Theorem~\ref{vertical-bnd-intro}, even for $t=1$.

\begin{theorem}\label{exa-cor-d-n-thm-in-intro}
    Let $\mbb K$ be an algebraically closed field of characteristic zero.
    Let $n, d\in \NN$ such that $d\mid n$ and $d\neq n$.
    There exists a log Calabi-Yau fibration $(X, B)\to Z$ such that
    \begin{itemize}
        \item $X$ is a normal projective threefold over $\mbb K$,
        \item $Z$ is a normal projective surface over $\mbb K$,
        \item a general fibre of $X\to Z$ is a nonsingular genus one curve over $\mbb K$, and
        \item there is an irreducible curve $C$ on $Z$ such that 
        \begin{itemize}
            \item $B$ has $n/d$ horizontal$/C$ irreducible components,
            \item every horizontal$/C$ irreducible component of $B$ has coefficient one in $B$, and
            \item for every horizontal$/C$ irreducible component $S$ of $B$, $\sdeg (S^{\nor}/C) = d$, where $S^{\nor}$ is the normalisation of $S$.
        \end{itemize}
    \end{itemize}
\end{theorem}


\mni
{\textbf{\sffamily{Sketch of the construction.}}}
The proof of Theorem~\ref{exa-cor-d-n-thm-in-intro} relies on the arithmetic results of 
\cite{LLR04} concerning reduction types of genus one curves.
In particular, for every $d\in \NN$, 
\cite{LLR04} constructs an arithmetic surface $g\colon X\to \spec \wh{R}$ (see \S\ref{elements-ari-surfaces}), where
\begin{itemize}
    \item $X$ is a regular scheme of dimension 2,
    \item $\wh{R}$ is the completion of a discrete valuation ring $R$,
    \item the residue field $k$ of $R$ admits a cyclic Galois extension $k'/k$ of degree $d$,
    \item the geometric generic fibre of $g$ is a smooth projective genus one curve, and
    \item the special fibre $X_k\to \spec k$ of $g$ is reduced and is a cycle of $\PP^1_{k'}$. 
\end{itemize}
See Proposition~\ref{ext-over-completion} for the detailed statement.
Let $K_X$ be the canonical divisor of $X$.  Then 
$K_X + X_k\sim 0/Z$, and $(X, X_k)$ is log canonical; cf. \cite[\S 2.1]{Kol_singularities_of_MMP}.
If $S_k$ denotes an irreducible component of $X_k$, 
then passing to the algebraic closure $\ol{k}$,
the base change $S_k\times_{\spec k}\spec \ol{k}$ splits into a disjoint union of $d$ projective lines over $\ol{k}$, 
hence $\sdeg (S_k/\spec k) = d$.

However, the construction above takes place over the Dedekind scheme $\spec \wh{R}$,
which is not of finite type over an algebraically closed field, 
so it does not directly give a log Calabi-Yau fibration.
To overcome this, we proceed in three stages.

\medskip

\emph{Step 1}. Let $R$ be the localisation ring of $\AAA_{\mbb K}^2$ at the generic point of a line.
Then $R$ is a discrete valuation ring that is
essentially of finite type over $\mbb K$.
Using Artin's algebraic approximation theory \cite{Artin},
we approximate the arithmetic surface $X\to \spec \wh{R}$ by a fibred surface $X'\to \spec R'$ 
(see \S\ref{elements-ari-surfaces}), where $R'$ is a discrete valuation ring admitting an \'etale
homomorphism $R\to R'$, and the residue field of $R'$ is isomorphic to that of $R$.
The approximation process is chosen in a certain careful way
to ensure that the singularity and log Calabi-Yau properties are preserved;
for more details, see Theorem~\ref{exa-thm-d-n}.

\medskip

\emph{Step 2}. Since $R'$ is also essentially of finite type over $\mbb K$,
it can be realised as the localisation of 
a nonsingular quasi-projective surface $Z^{\circ}$ 
at the generic point of a nonsingular curve $C^{\circ}\subset Z^{\circ}$.
Thus, we can assume that $X'\to \spec R'$ spreads out to 
a log Calabi-Yau fibration $(X^{\circ}, B^{\circ})\to Z^{\circ}$
satisfying the properties in Theorem~\ref{exa-cor-d-n-thm-in-intro}.
This process is formalised in Theorem~\ref{exa-cor-d-n},
which can be viewed as a quasi-projective version of Theorem~\ref{exa-cor-d-n-thm-in-intro}.

\medskip

\emph{Step 3}. Finally, we take a log canonical projective compactification $(\ol{X}, \ol{B})\to \ol{Z}$
of $(X^{\circ}, B^{\circ})\to Z^{\circ}$ following \cite{Hacon-Xu-closure}.
After running an MMP, the resulting family $(X, B)\to Z$ is 
a log Calabi-Yau fibration $(X, B)\to Z$ with $X, Z$ projective,
which completes the proof of Theorem~\ref{exa-cor-d-n-thm-in-intro}.
This is treated in Theorem~\ref{proj-exa-cor-d-n}.


\medskip

\mni
{\textbf{\sffamily{Acknowledgements.}}}
The first author was supported by a grant from Tsinghua University and 
a grant of the National Program of Overseas High Level Talent. 
We thank Qing Liu for answering our questions about the results in \cite{LLR04}. 
We thank the anonymous referees for many helpful comments to improve this paper.


\section{Preliminaries}\label{preliminaries}

The set of natural numbers $\NN$ is the set of all positive integers, 
so it does not contain 0.  

\subsection{Algebraic schemes}
A \emph{scheme} is defined in the sense of \cite[\S II.2]{Hart};
all the schemes in this paper are Noetherian.  
For an integral scheme $X$, we denote by $X^{\nor}$ the \emph{normalisation} of $X$.
For normalisation of schemes that may not be integral, 
see \cite[Definition~7.5.1]{Liuqing}.

Let $k$ be an arbitrary field.
By an \emph{algebraic scheme over $k$} (or a \emph{$k$-scheme}), 
we mean a scheme that is of finite type and separated over $\spec k$.
A \emph{variety} is a quasi-projective, reduced and irreducible $k$-scheme
for an algebraically closed field $k$ of characteristic zero;
when emphasising the quasi-projectivity, we also call it a \emph{quasi-projective variety}.


\subsection{Contractions}
 
A \emph{contraction} is a projective morphism of schemes $f\colon X\to Y$
such that $f_*\mc{O}_X = \mc{O}_Y$; $f$ is not necessarily birational (see \cite[\S 2.1]{B-Fano}).
In particular, $f$ has geometrically connected fibres.  Moreover, if $X$ is normal, then $Y$ is also normal.


\subsection{Divisors}\label{divisors-pre}
Let $X$ be a scheme.  A \emph{prime divisor} on $X$ is a reduced and irreducible closed subscheme of $X$
of codimension one.  By a \emph{divisor}, we mean a \emph{Weil divisor} on $X$, that is,
a finite $\ZZ$-linear combination of prime divisors on $X$, 
so a divisor is not necessarily Cartier.
For $\QQ$-linear equivalence (respectively, $\RR$-linear equivalence) of 
$\QQ$-divisors (respectively, of $\RR$-divisors), see \cite[\S 2.3]{B-Fano}.

Let $f\colon X\to Z$ be a morphism of schemes, and let $D$ be a nonzero $\RR$-divisor on $X$.
We say that $D$ is \emph{vertical$/Z$} if $f(\supp D)$ does not contain
generic point of any irreducible component of $Z$.
If $D$ does not have any vertical$/Z$ irreducible components,
we say that $D$ is \emph{horizontal$/Z$}.


\subsection{Pairs and singularities}
We will use standard notions and results from the Minimal Model Program 
\cite{Kollar92et-al, km98, BCHM, Kol_singularities_of_MMP, Fujino-book}.
Throughout this paper, a \emph{pair} $(X, B)$ consists of a normal quasi-projective variety $X$ and an
$\RR$-divisor $B\ge 0$ such that $K_X + B$ is $\RR$-Cartier.
For \emph{log canonical}, \emph{Kawamata log terminal} pairs (\emph{lc}, \emph{klt} pairs for short),
see \cite[\S 2.3]{km98}.  
A \emph{semi-log canonical pair} $(X, B)$ (or an \emph{slc} pair, to abbreviate) consists of 
a quasi-projective scheme $X$ of pure dimension
over an algebraically closed field of characteristic zero,
and an $\RR$-divisor $B\ge 0$ satisfying the following conditions:
\begin{itemize}
    \item $X$ is $S_2$ whose codimension one points are either regular points or nodes,
    \item no component of $\supp B$ is contained in the singular locus of $X$,
    \item $K_X + B$ is $\RR$-Cartier, and
    \item if $\pi\colon X^{\nor} \to X$ is the normalisation of $X$, and if $B^{\nor}$ is the sum of the birational transform of $B$ and the conductor divisor of $\pi$, then $(X^{\nor}, B^{\nor})$ is lc.
\end{itemize}
For more details about slc pairs, see \cite[Chap. 5]{Kol_singularities_of_MMP}.


\subsection{Arithmetic genus}

Let $k$ be a field, and let $C$ be a projective scheme over $\spec k$,
whose irreducible components are all of dimension one.
Denote by $\chi_k(\mc{O}_C)$ the Euler-Poincar\'e characteristic of $\mc{O}_C$ (over $k$).
The \emph{arithmetic genus} of $C$ is defined to be the integer $p_a(C) \coloneqq 1 - \chi_k(\mc{O}_C)$.


\subsection{Elements of arithmetic surfaces}\label{elements-ari-surfaces}

As clarifications, we collect some basic notions about arithmetic surfaces from \cite{Liuqing}.
A \emph{Dedekind scheme} is a normal Noetherian scheme of dimension one; cf. \cite[Definition~4.1.2]{Liuqing}.

Let $S$ be a Dedekind scheme.
We call an integral, projective, flat $S$-scheme $\pi\colon X\to S$ of dimension two a
\emph{fibred surface over $S$}.  We also call $X$ a \emph{projective flat $S$-curve}.  
We say that $\pi\colon X\to S$ is a \emph{normal} (respectively, \emph{regular}) \emph{fibred surface}
if $X$ is normal (respectively, regular).
A \emph{morphism} (respectively, a \emph{rational map}) between fibred surfaces over $S$
is a morphism (respectively, a rational map) that is compatible 
with the structure of $S$-schemes.  For more details, see \cite[\S 8.3]{Liuqing}.

Let $S$ be a Dedekind scheme.
An \emph{arithmetic surface} (over $S$) is a regular fibred surface $\pi\colon X\to S$; 
cf. \cite[Definition~8.3.14]{Liuqing}.
We say that an arithmetic surface $X\to S$ is \emph{relatively minimal}
if every birational morphism of arithmetic surfaces $X\to Y$ (over $S$)
is an isomorphism.  We say that $X\to S$ is \emph{minimal} if every birational 
map of arithmetic surfaces $Y\dashrightarrow X$ (over $S$) is a birational morphism.
For more details, see \cite[\S 9.3.2]{Liuqing}.
It is evident that a minimal arithmetic surface is relatively minimal.
The converse also holds if the generic fibre of the arithmetic surface has 
arithmetic genus $\ge 1$; see \cite[Corollary~9.3.24]{Liuqing}.


\section{Unbounded strong Stein degree}\label{pf-unbnd}

\subsection{Split multiplicative reduction of genus one curves}

Our proof of Theorem~\ref{exa-cor-d-n-thm-in-intro} is based on the following result
about reduction types of genus one curves from \cite{LLR04}.

\begin{proposition}[cf. \protect{\cite[\S 8]{LLR04}}]\label{ext-over-completion}
    Let $n, d\in \NN$ such that $d\mid n$ and $d\neq n$.
	Let $R$ be a complete discrete valuation ring
	with fraction field $K$ and residue field $k$ respectively. 
    Assume that $k$ admits a cyclic Galois extension $k'/k$ of degree $d$.
	Then there exists a relatively minimal arithmetic surface $X\to \spec R$ such that
	\begin{itemize}
		\item [\emph{(1)}] the generic fibre $X_K\to \spec K$ of $X\to \spec R$ is a geometrically integral, smooth, projective curve of arithmetic genus one, and
		\item [\emph{(2)}] the special fibre $X_k$ satisfies that
		      \begin{itemize}
                  \item [\emph{(2.a)}] $X_k$ is reduced, 
		          \item [\emph{(2.b)}] every irreducible component of $X_k$ is isomorphic to $\PP_{k'}^1$,
		          \item [\emph{(2.c)}] $X_k$ is a cycle of $\PP_{k'}^1$ consisting of $n/d$ irreducible components,
		          \item [\emph{(2.d)}] the residue field of every intersection point between irreducible components of $X_k$ is isomorphic to $k'$, and
                  \item [\emph{(2.e)}] the geometric fibre $X_{\bar{k}}$ is a reduced cycle of $n$ projective lines over $\bar{k}$, that is, a reduced and connected nodal curve of genus one, which has nonsingular rational irreducible components such that every irreducible component is connected to other two irreducible components at two distinct closed points.
		      \end{itemize}
	\end{itemize}
\end{proposition}

\begin{proof}
    As $k$ has a cyclic Galois extension $k'/k$ of degree $d$, by \cite[Theorem~3, page~70]{Lorenz-alg-II},
    there is a finite, Galois, and unramified extension $K'$ of $K$ of degree $d$ such that
    \begin{itemize}
        \item the residue field of $K'$ is $k'$, and
        \item the Galois group $\galois (K'/K)$ is canonically isomorphic to $\galois (k'/k)$.
    \end{itemize}
	Then the existence of $X\to \spec R$ 
    follows immediately from Proposition~8.3, Proposition~8.7, and Example~8.8 in \cite{LLR04}.
    Note that the cyclicity of $K'/K$ is required to apply \cite[Proposition~8.7]{LLR04}.
	(We remark that the irreducible components of the $X_k$ in \cite[Example~8.8]{LLR04} are not geometrically irreducible,
	which is a typo in \cite[Example~8.8]{LLR04}.)
    Moreover, by \cite[Propositions~8.3,~8.7]{LLR04}, $X_K/K$ is trivialised by $K'$ and its 
    Jacobian has split multiplicative reduction of type $I_n$,
    which shows that $X_{\bar{k}}$ is a reduced cycle of $n$ projective lines over $\bar{k}$.
\end{proof}


\subsection{Analyticity and Algebraicity}
By approximating the complete DVR in Proposition~\ref{ext-over-completion} in a suitable way,
we get the following algebraic version of Proposition~\ref{ext-over-completion}.

\begin{theorem}\label{exa-thm-d-n}
	Let $n, d\in \NN$ such that $d\mid n$ and $d\neq n$.
	Let $R$ be an excellent DVR whose residue field $k$ has a cyclic Galois 
    extension $k'/k$ of degree $d$.  Then there exist
	\begin{itemize}
		\item [\emph{(1)}] a discrete valuation ring $R'$ admitting an \'etale morphism $\spec R'\to \spec R$ such that $R$ and $R'$ have isomorphic residue fields (also denoted by $k$), and
		\item [\emph{(2)}] a normal scheme $X'$ of dimension 2 admitting a projective morphism 
              \[ g'\colon X'\to \spec R' \] 
              such that
        \begin{itemize}
            \item [\emph{(2.1)}] the generic fibre of $g'$ satisfies the property in Proposition~\ref{ext-over-completion} (1),
            \item [\emph{(2.2)}] the special fibre $X'_k$ of $g'$ satisfies all the properties in Proposition~\ref{ext-over-completion} (2),
            \item [\emph{(2.3)}] the relative dualising sheaf $\omega_{X'/\spec R'}$ of $g'$ is an invertible sheaf on $X'$, and
            \item [\emph{(2.4)}] the relative dualising sheaf $\omega_{X_k'/\spec k}$ of the special fibre $X_k'\to \spec k$ satisfies that $\omega_{X'_k/\spec k} \cong \omega_{X'/\spec R'}|_{X'_k}$.
        \end{itemize}
	\end{itemize}
\end{theorem}

\begin{proof}
    \emph{Step 1}.
    Denote by $\mf{m}_R$ the maximal ideal of $R$.
    Let $\wh{R}$ be the $\mf{m}_R$-adic completion of $R$, which is a complete DVR.
	Notice that the canonical homomorphism $R\to \wh{R}$ induces an isomorphism of residue fields,
	hence by abusing of notation, we denote by $k$ the residue fields of $R$ and $\wh{R}$.
    Denote by $S$ the Dedekind scheme $\spec R$, and by $\wh{S}$ the scheme $\spec \wh{R}$.
    By assumption, $k'/k$ is a cyclic Galois extension of degree $d$.
	Then by Proposition~\ref{ext-over-completion}, there exists a relatively minimal arithmetic surface
	\[ \hat{g}\colon \wh{X}\to \wh{S} \]
    satisfying all the properties listed in Proposition~\ref{ext-over-completion}.  
    As $\hat{g}\colon \wh{X}\to \wh{S}$ is projective, 
    there exists a closed immersion of $\wh{S}$-schemes
    \[ \wh{X}\to \PP_{\wh{S}}^N \] 
    for some $N\in \NN$ sufficiently large.
	Denote by $\Phi$ the Hilbert polynomial of fibres of $\hat{g}$ with respect to 
	the embedding $\wh{X}\to \PP_{\wh{S}}^N$.
	Denote by $\hilb_{\Phi}(\PP^N_S/S)$ the Hilbert scheme of subschemes of $\PP^N_S$ over $S$
    with Hilbert polynomial $\Phi$.  Then $\hilb_{\Phi}(\PP^N_S/S)$ is a projective $S$-scheme; cf. 
	\cite[Chap. 5]{FGA-explained} and \cite[Theorem~I.1.4]{Kol-rc}.  Denote by 
    \[ g\colon \Sigma\to \hilb_{\Phi}(\PP^N_S/S)\]  
    the universal family over $\hilb_{\Phi}(\PP^N_S/S)$.
	Then there exists a unique $S$-morphism 
    \[ \wh{\varphi}\colon \wh{S}\to \hilb_{\Phi}(\PP^N_S/S) \] 
    such that $\hat{g}\colon \wh{X}\to \wh{S}$ is the base change of $g$ via $\wh{\varphi}$.
	\[\xymatrix{
	\wh{X}\ar[d]_{\hat{g}}\ar[r] & \Sigma\ar[d]^g\ar[r] & \PP^N_S\ar[d] \\
	\wh{S}\ar[r]^-{\wh{\varphi}} & \hilb_{\Phi}(\PP^N_S/S)\ar[r] & S  
	}\]
	Denote by $\wh{H}$ the scheme-theoretic image of $\wh{\varphi}$, 
	which is an integral closed subscheme of $\hilb_{\Phi}(\PP^N_S/S)$.  Let 
    \[ \ol{\varphi}\colon \wh{S}\to \wh{H} \] 
    be the induced $\hilb_{\Phi}(\PP^N_S/S)$-morphism, which is dominant.  

    \medskip

    \emph{Step 2}.
    Denote by $\hat{s}$ the closed point of $\wh{S}$, and by $\hat{q} \coloneqq \ol{\varphi}(\hat{s})$.
    Let $\wh{\Sigma}$ be the base change of $\Sigma$ via $\wh{H}\to \hilb_{\Phi}(\PP^N_S/S)$.
    Then the generic fibre of $\wh{\Sigma}\to \wh{H}$ is a geometrically integral, smooth, projective curve 
    of arithmetic genus one as the generic fibre of $\hat{g}$ is so.
    
    Let $U\subseteq \wht{H}$ be an affine open subscheme containing $\hat{q}$
    such that every fibre of $\wh{\Sigma}\to \wh{H}$ over $U$ is a geometrically reduced,
    geometrically connected, projective curve of arithmetic genus one;
    see \cite[Th\'eor\`eme~(12.2.4)]{EGA-IV-3}.
    We can assume that $U = \spec A$, where $A$ is a finitely generated $R$-algebra.
    It is evident that $\ol{\varphi}$ factors through the inclusion $U\to \wht{H}$.  Denote by 
    \[ \varphi_U\colon \wht{S}\to U \] 
    the induced morphism of $S$-schemes.

    Denote by $\Sigma_U$ the base change of $\wh{\Sigma}$ via $U\to \wh{H}$.
    Let $U_{\sing}$ be the maximal reduced closed subscheme of $U$
    over which every fibre of $\Sigma_U\to U$ is singular.  As $U$ is affine, 
    we can assume that $U_{\sing} = V(I)$ for an ideal $I\subset A$.
    Note that the generic fibre of $\Sigma_U\to U$ is smooth, hence $U_{\sing}$ is a proper closed subset of $U$.
    Shrinking $U$ near $\hat{q}$, we can assume that every irreducible component of $U_{\sing}$ contains $\hat{q}$.
    
    Let $\Set{\alpha_1,\dots,\alpha_r}$ be a set of generators of the ideal $I$.
    As $\varphi_U\colon \wh{S}\to U$ is dominant, we see that the corresponding ring homomorphism
    \[ \varphi_U^{\#}\colon A\to \wh{R} \]
    is injective.  In particular, the image $\beta_j \coloneqq \varphi_U^{\#}(\alpha_j)\in \wh{R}$ of every $\alpha_j$ is nonzero.
    Let $\mf p$ be the prime ideal of $A$ corresponding to $\hat{q}\in U$, that is,
    \[ \mf p = (\varphi_U^{\#})^{-1}(\mf m_R \wh{R}). \]
    Since every irreducible component of $U_{\sing}$ contains $\hat{q}$,
    we see that $\alpha_j\in \mf p$ for every $1\le j\le r$.
    Thus, every $\beta_j$ is a nonzero element of $\mf{m}_R\wh{R}$.
    By Krull Intersection Theorem (see \cite[Theorem~8.10]{CA-Mat}), we have
    \[ \bigcap_{a\in \NN} \mf{m}_R^a\wh{R} = (0). \]
    Then there exists a natural number $c\gg 0$ such that $\beta_j\not\in \mf{m}_R^c\wh{R}$
    for every $1\le j\le r$.
    
	Since $U$ is of finite type over $S$, by \cite[Corollary~(2.5)]{Artin},
	there exist an \'etale neighbourhood $S'\to S$ of the closed point of $S$ 
    (see \cite[page 27]{Artin}) and an $S$-morphism 
	\[ \varphi'\colon S' \to U \]
	such that
    \begin{equation}
        \varphi_U \equiv \varphi' \Mod{\mf{m}_R^c}.\label{equiv-mor} \tag{$\star$}
    \end{equation}
    Here we have applied the assumption that $R$ is excellent as this is required in \cite{Artin}.
    We can assume that $S'$ has a single closed point, hence  
    we can write $S' = \spec R'$ for a discrete valuation ring $R'$.  
    Moreover, by the definition of \'etale neighbourhoods \cite[page 27]{Artin}, 
    the residue field of $R'$ is isomorphic to $k$.
    By \cite[page 23]{Artin}, \eqref{equiv-mor} means that the composite morphisms of schemes
    \[ \wh{S}\times_S \spec (R/\mfm_R^c)\to \wh{S}\xrightarrow{\varphi_U} U \,\text{ and }\, 
    S'\times_S\spec (R/\mfm_R^c)\to S'\xrightarrow{\varphi'} U \]
    are identical.  Denote by 
    \[ (\varphi')^{\#}\colon A\to R' \] 
    the ring homomorphism corresponding to the $S$-morphism $\varphi'$.  
    Then \eqref{equiv-mor} is also equivalent to saying that the composite ring homomorphisms
    \[ A\xrightarrow{\varphi_U^{\#}} \wh{R} \to \wh{R}/\mf{m}_R^c\wh{R}\,\text{ and }\,
    A\xrightarrow{(\varphi')^{\#}} R'\to R'/\mf{m}_R^c R' \]
    are identical.

    Since the image of every $\beta_j$ in $\wh{R}/\mf{m}_R^c\wh{R}$ is nonzero,
    we see that the image of every $\alpha_j$ in $R'/\mf{m}_R^c R'$ is also nonzero.
    Hence for every $1\le j\le r$, the image $\gamma_j \coloneqq (\varphi')^{\#}(\alpha_j)\in R'$ is nonzero.
    If the image of $\varphi'$ is contained in $U_{\sing}$,
    then $\varphi'$ factors through the closed immersion $U_{\sing}\to U$.
    In this case, every $\gamma_j$ is zero in $R'$.  Thus, we can conclude that 
    $\varphi'(\eta_{S'})\in U\setminus U_{\sing}$, where $\eta_{S'}$ is the generic point of $S'$.

    \medskip

    \emph{Step 3}.  
    Denote by $g'\colon X'\to S'$ the base change of $\Sigma_U\to U$
    via the $S$-morphism $\varphi'\colon S'\to U$.
    By Step 2, the generic fibre of $g'$ is a geometrically integral, smooth, projective curve
    of arithmetic genus one.
    As $g'$ is flat, by \cite[Proposition~4.3.8]{Liuqing}, $X'$ is an irreducible and reduced scheme.
    Hence $g'\colon X'\to S'$ is a fibred surface over $S'$.

    By Step 2, $g'$ and $\hat{g}$ have the same special fibre over $k$.
    By Proposition~\ref{ext-over-completion}, the geometric fibre $X'_{\bar{k}}$ is a reduced cycle of 
    $\PP_{\bar{k}}^1$.  Thus, $X'_k$ is a semi-stable curve over $k$; see \cite[Definition~10.3.2]{Liuqing}.
    Then $g'\colon X'\to S'$ is a semi-stable curve over $S'$; see \cite[Definition~10.3.14]{Liuqing}.
    Since $S'$ is a Dedekind scheme, and since the generic fibre of $g'$ is smooth,
    by \cite[Proposition~10.3.15]{Liuqing}, we see that $X'$ is a normal scheme and that
    $g'\colon X'\to S'$ is a local complete intersection.
    By \cite[\S 6.4.2]{Liuqing} and \cite[Theorem~6.4.32]{Liuqing}, the relative dualising sheaf $\omega_{X'/S'}$ exists, 
    which is an invertible sheaf on $X'$.  Moreover, as $g'$ is flat, \cite[Theorem~6.4.9]{Liuqing} shows that
    \[ \omega_{X'_k/\spec k} \cong \omega_{X'/S'}|_{X'_k}, \]
    where $\omega_{X'_k/\spec k}$ is the relative dualising sheaf 
    of the projective morphism $X'_k\to \spec k$.
    Therefore, the normal fibred surface $g'\colon X'\to S'$ satisfies all the properties as desired.  
\end{proof}


\subsection{Proof of Theorem~\ref{exa-cor-d-n-thm-in-intro}}

By spreading out the family of genus one curves obtained in Theorem~\ref{exa-thm-d-n},
we get a quasi-projective version of Theorem~\ref{exa-cor-d-n-thm-in-intro}.

\begin{theorem}\label{exa-cor-d-n}
    Let $\mbb K$ be an algebraically closed field of characteristic zero.
    Let $n, d\in \NN$ such that $d\mid n$ and $d\neq n$.
    There exists a log Calabi-Yau fibration $(X, B)\to Z$ such that
    \begin{itemize}
        \item $X$ is a normal quasi-projective threefold over $\mbb K$,
        \item $Z$ is a smooth quasi-projective surface over $\mbb K$,
        \item a general fibre of $X\to Z$ is a nonsingular genus one curve over $\mbb K$,
        \item $B$ is a reduced Cartier divisor on $X$ with $n/d$ nonsingular irreducible components,
        \item there is a smooth curve $C$ on $Z$ that is the image of every irreducible component of $B$ in $Z$, and
        \item for every irreducible component $S$ of $B$, $\sdeg (S/C) = d$.
    \end{itemize}
\end{theorem}

\begin{proof}
    \emph{Step 1}.
    Let $L$ be a line in the affine surface $\AAA^2_{\mbb K}$.
    Denote by $\eta$ the generic point of $L$, and by
    $R$ the localisation ring of $\AAA^2_{\mbb K}$ at $\eta$, which is a DVR.  
    Denote by $k$ the residue field of $R$, then $k$ is isomorphic to 
    $\mbb K (t)$, where $t$ is a transcendental variable over $\mbb K$.
    Since $\mbb K$ is algebraically closed, it contains a primitive 
    $d$-th root of unity.  Then $k \cong \mbb K (t)$ admits a cyclic Galois 
    extension $k'/k$ of degree $d$; cf. \cite[Proposition~36, page~625]{Dummit-F-3rd-algebra}.
    Thus, by Theorem~\ref{exa-thm-d-n}, there exist a discrete valuation ring $R'$, 
    an \'etale morphism $\spec R'\to \spec R$, and
    a normal fibred surface 
    \[ g'\colon X'\to \spec R' \]
    such that all the properties in Theorem~\ref{exa-thm-d-n} are satisfied.
    Note that $k$ is also the residue field of $R'$,
    and $X'_k$ is a reduced cycle of $\PP_{k'}^1$ consisting of $n/d$ irreducible components.

    Let $K'$ be the fraction field of $R'$.  Let 
    \[ \mu\colon Z'\to \AAA^2_{\mbb K} \] 
    be the normalisation of $\AAA^2_{\mbb K}$ in $K'$,
    which is a finite morphism of normal varieties; see \cite[Propositions~4.1.25,~4.1.27]{Liuqing}.
    Particularly, by \cite[Corollaire~(18.12.4)]{EGA-IV-4}, $\mu$ is proper.
    As $K(Z') \cong K'$, the valuative criterion of properness shows that there is a unique morphism    
    $\nu\colon \spec R' \to Z'$ fitting into the commutative diagram
    \[\xymatrix{
    \spec K'\ar[rr]\ar[d] & & Z'\ar[d]^{\mu} \\
    \spec R'\ar[r]\ar[rru]^{\nu} & \spec R\ar[r] & \AAA_{\mbb K}^2
    }\]
    where the top horizontal arrow is induced by $K(Z')\cong K'$.
    Denote by $\eta'$ the image of the generic point of $\spec R'$ under $\nu$.
    As $\mu$ is finite, we see that $\eta'$ is a codimension one point of $Z'$.
    By \cite[Lemma~3.3.24]{Liuqing}, $R'\cong \mc{O}_{Z', \eta'}$. 
    Then there exist 
    \begin{itemize}
        \item an affine, nonsingular, open subvariety $Z\subseteq Z'$ containing $\eta'$, and 
        \item a normal variety $X$ equipped with a flat, projective morphism 
              \[ f\colon X\to Z \] 
              such that (cf. \cite[Proposition~(7.8.6)]{EGA-IV-II})
        \[ X'\cong X\times_Z \spec \mc{O}_{Z', \eta'}.  \]
    \end{itemize}
    In particular, $X$ is a normal quasi-projective threefold over $\mbb K$.

    \medskip

    \emph{Step 2}.
    Denote by $C$ the closure of $\eta'$ in $Z$, which is a nonsingular curve over $\mbb K$ 
    (up to shrinking $Z$ near $\eta'$).  Let $B$ be the fibre product $X\times_Z C$.
    Then by \cite[Proposition~4.3.8]{Liuqing}, $B$ is a reduced Cartier divisor on $X$.  
    Note that the generic fibre of $B\to C$ is equal to the special fibre of $g'\colon X'\to \spec R'$.
    As every irreducible component of $X'_k$ is
    isomorphic to $\PP_{k'}^1$, every irreducible component of the generic fibre of $B\to C$ is nonsingular.
    Thus, up to shrinking $Z$ near $\eta'$, we can assume that $B$
    has $n/d$ nonsingular irreducible components.
    It is evident that $B\sim 0/Z$.

    By Theorem~\ref{exa-thm-d-n}, the geometric generic fibre of $B\to C$ is a cycle of projective lines.
    Then by \cite[II, Proposition~1.5]{Deligne-Rapoport}, shrinking $Z$ near $\eta'$, 
    we can assume that every closed fibre of $B\to C$ is a cycle of projective lines.
    We claim that $B$ has slc singularities.  To this end, notice that
    $B\to C$ is a split semi-stable curve over $C$ in the sense of \cite[2.21, 2.22]{dejong-alteration}.
    Pick an arbitrary singular closed point $b\in B$, and let $c\in C$ be its image.
    By \cite[2.23]{dejong-alteration}, we have
    \[ \wh{\mcO}_{B,b}\cong \wh{\mcO}_{C,c}[[u,v]]/(uv - h) \]
    for some $h\in \wh{\mcO}_{C,c}$.  As the generic fibre of $B\to C$ is singular, we see that $h = 0$.
    Hence $b$ is a double normal crossing point of $B$.  
    Thus, $(B, 0)$ is an slc pair; see \cite[\S 5.2]{Kol_singularities_of_MMP}
    and \cite[\S 1.4]{Kollar-g-type-book}.

    \medskip

    \emph{Step 3}.
    By Theorem~\ref{exa-thm-d-n}, $\omega_{X'/\spec R'}$ is an invertible sheaf on $X'$.
    Thus, we can also assume that $\omega_{X/Z}$ is an invertible sheaf.
    As $Z$ is affine and nonsingular, we can conclude that $K_X\sim K_{X/Z}$ is a Cartier divisor.
    Moreover, by \cite[II, Proposition~1.6]{Deligne-Rapoport}, 
    we see that $K_X \sim 0/Z$.  Therefore, we get
    \[  K_X + B \sim 0/Z. \]
    By adjunction (see \cite[Proposition~4.5]{Kol_singularities_of_MMP}), we can write
    \[ (K_X + B)|_B = K_B. \]
    As $B$ is slc, inversion of adjunction shows that $(X, B)$ is lc; 
    see \cite[Theorem~4.9, Definition–Lemma~5.10]{Kol_singularities_of_MMP}.
    Then $(X, B)\to Z$ is a log Calabi-Yau fibration.
    Pick an arbitrary irreducible component $S$ of $B$.
    Since the generic fibre of $S\to C$
    is isomorphic to $\PP_{k'}^1$ (as schemes over $\spec k$), the geometric generic fibre
    of $S\to C$ is isomorphic to a disjoint union of $d$
    projective lines over $\bar k$, which shows that $\sdeg (S/C) = d$.
\end{proof}


By taking log canonical projective compactification of the log Calabi-Yau fibration 
in Theorem~\ref{exa-cor-d-n} and running an MMP,
we get the proof of Theorem~\ref{exa-cor-d-n-thm-in-intro}.

\begin{theorem}[=Theorem~\ref{exa-cor-d-n-thm-in-intro}]\label{proj-exa-cor-d-n}
    Let $\mbb K$ be an algebraically closed field of characteristic zero.
    Let $n, d\in \NN$ such that $d\mid n$ and $d\neq n$.
    There exists a log Calabi-Yau fibration $(X, B)\to Z$ such that
    \begin{itemize}
        \item $X$ is a normal projective threefold over $\mbb K$,
        \item $Z$ is a normal projective surface over $\mbb K$,
        \item a general fibre of $X\to Z$ is a nonsingular genus one curve over $\mbb K$, and
        \item there is an irreducible curve $C$ on $Z$ such that 
        \begin{itemize}
            \item $B$ has $n/d$ horizontal$/C$ irreducible components,
            \item every horizontal$/C$ irreducible component of $B$ has coefficient one in $B$, and
            \item for every horizontal$/C$ irreducible component $S$ of $B$, $\sdeg (S^{\nor}/C) = d$, where $S^{\nor}$ is the normalisation of $S$.
        \end{itemize}
    \end{itemize}
\end{theorem}

\begin{proof}
    For the given natural numbers $d,n$,
    let $(X^{\circ}, B^{\circ}) \to Z^{\circ}$ be a log Calabi-Yau fibration 
    satisfying all the properties listed in Theorem~\ref{exa-cor-d-n}.
    Let $C^{\circ}\subset Z^{\circ}$ be the smooth curve that is the image of every irreducible component of $B^{\circ}$.
    
    Let $\ol{Z}$ be a projective compactification of $Z^{\circ}$.
    As $Z^{\circ}$ is nonsingular, by taking resolution, 
    we can assume that $\ol{Z}$ is a nonsingular projective surface.
    By \cite[Corollary~1.2]{Hacon-Xu-closure}, there exist a projective morphism $\ol{X}\to \ol{Z}$
    and an lc pair $(\ol{X}, \ol{B})$ such that $X^{\circ}$ is an open subset of $\ol{X}$, 
    $X^{\circ} = \ol{X}\times_{\ol{Z}}Z^{\circ}$, and $B^{\circ} = \ol{B}|_{X^{\circ}}$.
    Taking the Stein factorisation of $\ol{X}\to \ol{Z}$, we can also assume that
    $\ol{X}\to \ol{Z}$ is a contraction of normal projective varieties.
    
    It is evident that $K_{\ol{X}} + \ol{B}$ is pseudo-effective over $\ol{Z}$.
    As $\dim \ol{X} = 3$, we can run $\phi\colon \ol{X} \dashrightarrow X$, a $(K_{\ol{X}} + \ol{B})$-MMP over $\ol{Z}$;
    see \cite{Kollar92et-al, km98, Fujino-book}.
    Denote by $B$ the pushdown of $\ol{B}$ to $X$. 
    Then $(X, B)$ is also an lc pair, and $K_X + B$ is nef$/\ol{Z}$.
    Moreover, by the proof of Theorem~\ref{exa-cor-d-n}, we see that
    $K_{X^{\circ}} + B^{\circ} \sim 0/Z^{\circ}$.  Hence by the cone theorem
    (see \cite{km98, Fujino-book}), $\phi$ is an isomorphism over $Z^{\circ}$.
    
    By the log abundance theorem for threefolds (see \cite{Kollar92et-al, KMM94, Fujino-book}),
    $K_X + B$ is semi-ample over $\ol{Z}$.  Hence there are contractions of normal projective varieties 
    $f\colon X\to Z$ and $Z\to \ol{Z}$ such that $f$ is a $\ol{Z}$-morphism, and 
    $K_X + B = f^*A$ for an ample$/\ol{Z}$ divisor $A$ on $Z$.
    Then $(X, B)\to Z$ is a log Calabi-Yau fibration.
    \[\xymatrix{
    (X^{\circ}, B^{\circ})\ar[r]\ar[d] & (\ol{X}, \ol{B})\ar@{-->}[r]^-{\phi}\ar[d] & (X, B)\ar[d]^f \\
    Z^{\circ}\ar[r] & \ol{Z} & Z\ar[l]
    }\]
    As $\phi$ is an isomorphism over $Z^{\circ}$, $K_X + B$ is numerically trivial over the open subset $Z^{\circ}$.
    Thus, we see that $Z\to \ol{Z}$ is birational.
    Denote by $\ol{C}$ the closure of $C^{\circ}$ in $\ol{Z}$.
    As the image $E\subset \ol{Z}$ of exceptional locus of $Z\to \ol{Z}$ has codimension $2$ in $\ol{Z}$,
    $\ol{C}$ is not contained in $E$.
    Denote by $C$ the birational transform of $\ol{C}$ to $Z$, which is an irreducible curve on $Z$.
    Since $\phi$ is an isomorphism over $Z^{\circ}$, and since $Z\to \ol{Z}$
    is an isomorphism over the generic point of $C^{\circ}$,
    we can conclude that $B$ has $n/d$ horizontal$/C$ irreducible components,
    and every such irreducible component has coefficient one in $B$.
    Moreover, for every horizontal$/C$ irreducible component $S$ of $B$,
    as $\sdeg (S^{\nor}/C)$ is equal to the number of irreducible components
    of general fibres of $S\to C$ (see \cite[\S 2]{birkar-Q-25}), we see that $\sdeg (S^{\nor}/C) = d$.
\end{proof}


\medskip

\printbibliography
 
\vspace{1em}
 
\noindent\small{Caucher Birkar} 

\noindent\small{\textsc{Yau Mathematical Sciences Center, Tsinghua University, Beijing, China} }

\noindent\small{Email: \texttt{birkar@mail.tsinghua.edu.cn}}

\vspace{1em}
 
\noindent\small{Santai Qu} 

\noindent\small{\textsc{Institute of Geometry and Physics, University of Science and Technology of China, Hefei, Anhui Province, China} }

\noindent\small{Email: \texttt{santaiqu@ustc.edu.cn}}

\end{document}